\documentclass[10pt]{amsart}
\usepackage{amssymb}
\usepackage{amsfonts}
\usepackage{amsmath}
\usepackage[centertags]{amsmath}
\usepackage{newlfont}
\usepackage{url}

\newtheorem{theorem}[subsection]{Theorem}
\theoremstyle{plain}

\newtheorem{corollary}[subsection]{Corollary}

\newtheorem{definition}[subsection]{Definition}

\newtheorem{lemma}[subsection]{Lemma}
\newtheorem{fact}[subsection]{Fact}

\newtheorem{proposition}[subsection]{Proposition}

\numberwithin{equation}{subsection}

\theoremstyle{plain}

\theoremstyle{definition}

\theoremstyle{remark}

\theoremstyle{plain}

\input{tcilatex}

\begin{document}
\title{\textbf{On The }$E_{\alpha }$\textbf{-Envelopes of Hypercentral
Subgroups}}
\author{Tuba \c{C}akmak}
\address{Atat\"{u}rk University, Faculty of Science, Department of
Mathematics, 25240 Erzurum, Turkey }
\email{cakmaktuba@yahoo.com}

\begin{abstract}
The $E_{k}$ envelopes that generalize the double centralizers form a
descending chain. In this paper we show that this descending chain stops
after finitely many steps for hypercentral subgroups by defining the
transfinite forms of some basic descriptions. In particular, we prove that
the $E_{\alpha }$- envelopes of hypercentral subgroups are solvable in the
class of groups satisfying chain condition on centralizers. These extend
previous results on $E_{k}$ envelopes.
\end{abstract}

\thanks{This work was supported by T\"{U}B\.{I}TAK, the Scientific and
Technological Research Council of Turkey, through its programs 2214A and
2211E. }
\maketitle

\section{Introduction}

This paper continues a line of research in the footsteps of \cite{A-B} and 
\cite{Cakmak} and analyzes the properties of a technical tool, namely the $%
E_{k}$-envelopes introduced in \cite{A-B} to prove some definability
properties (in the sense of the first-order logic) in the class of $%
\mathfrak{M}_{C}$-groups, the groups that satisfy the descending chain
condition on centralizers, i.e. that do not have infinite descending chains
of centralizers of subsets. Several important classes of groups, of which
stable groups in model theory are a notable example, satisfy the descending
chain condition on centralizers. The introduction of the paper by Roger
Bryant (\cite{Bryant 1}) contains a detailed description of the basic
properties of $\mathfrak{M}_{C}$-groups.

\bigskip

In \cite{A-B}, Alt\i nel and Baginski showed that in an $\mathfrak{M}_{C}$%
-group, every nilpotent subgroup is contained in a definable subgroup of the
same nilpotency class. In doing this, they introduced special enveloping
subgroups of an arbitrary subgroup $H$, denoted $E_{k}\left( H\right) $ ($%
k\in 
\mathbb{N}
$). If $G$ is an arbitrary group, $H$ a subgroup of $G,$ then $E_{k}\left(
H\right) $ is a double centralizer of $H$ in a special section of $G$. For
every subgroup $H\leq G$, the $E_{k}\left( H\right) $ form a descending
chain. In \cite{Cakmak}, group theoretic and topological properties of $E_{k}
$ chains were analyzed. It was shown that if $G$ is an arbitrary group and $%
H\leq G$ is nilpotent the descending chain $\left( E_{k}\left( H\right)
\right) $ stabilizes after finitely many steps. This conclusion was based on
another that showed if $H$ is $k$-nilpotent then so is $E_{k}\left( H\right)
.$

\bigskip

In this paper we continue our investigation of these envelopes in a broader
context. We analyze the envelopes of hypercentral subgroups of arbitrary
groups and also of $\mathfrak{M}_{C}$-groups. This broader analysis
necessitates an ordinal-indexed version of our envelopes that we denote
using greek letters, the $E_{\alpha }$-envelopes. Their definition is the
natural continuation of the integer-indexed envelopes. In Theorem \ref%
{3.1.11}, we obtain a new finiteness condition.

\bigskip

The organization of the paper is as follows. In section 2, we revise various
tools. Section 3 is devoted to the $E_{\alpha }$-envelopes. In section 4, we
use the technical bases set up in section 3 to prove the main results of the
paper.

\bigskip

\section{Preliminaries}

In this section we will review the main facts required for the present
paper. Our notation is standard for basic group-theoretic notions: the
normalizer of any subset $H$ in $G$ is $N_{G}\left( H\right) =\left\{ g\in
G\mid \forall h\in H\text{ \ }g^{-1}hg\in H\right\} $; the centralizer of $H$
in $G$ is$\ C_{G}\left( H\right) =\left\{ g\in G\mid \forall h\in H\text{ \ }%
gh=hg\right\} $; $\left[ g,h\right] :=g^{-1}h^{-1}gh$ is the commutator of $%
g,h\in G$ elements; when $A,B\subseteq G$ we write$\ \left[ A,B\right]
:=\left\langle \left\{ \left[ a,b\right] \mid a\in A,b\in B\right\}
\right\rangle $. Also, we write $H\leq G$ to denote that $H$ is a subgroup
of $G$ and $H\trianglelefteq G$ to denote $H$ is normal in $G$. In
particular, $\mathfrak{M}_{C}$ denotes the class of groups satisfying the
minimal condition on centralizers.

\bigskip

We recall the definition of $E_{k}$\ envelopes, introduced in \cite{A-B}. \ 

\bigskip

\begin{definition}
Let $G$ be a group and $H$ a subgroup. For $k\in 
\mathbb{N}
,$ a sequence of subgroups $E_{k}\left( H\right) $ of $G$ is defined 
\begin{equation*}
E_{k}(H)=\left\{ g\in E_{k-1}(H)\mid \left[ g,C_{E_{k-1}(H)}^{k}\left(
H\right) \right] \leq C_{E_{k-1}(H)}^{k-1}\left( H\right) \right\}
\end{equation*}%
if $k>0$ and $E_{0}\left( H\right) =G.$
\end{definition}

It is clear that $E_{1}\left( H\right) =C_{G}\left( C_{G}\left( H\right)
\right) $. We remind a simple fact from \cite{Cakmak}:

\begin{fact}
\label{3.1}Let $G$ be a group and $H$ an abelian subgroup of $G.$Then $%
C_{G}\left( C_{G}\left( H\right) \right) $ is abelian.\bigskip
\end{fact}

\begin{definition}[\protect\cite{Bryant 1}]
Let $A$ be any subset of the group $G$. Set $C_{G}^{0}\left( A\right) =1$
and for $k\geq 1,$ the iterated centralizer of $A$ in $G$ is 
\begin{equation*}
C_{G}^{k}\left( A\right) =\left\{ x\in \underset{n<k}{\cap }N_{G}\left(
C_{G}^{n}\left( A\right) \right) \mid \left[ x,A\right] \subseteq
C_{G}^{k-1}\left( A\right) \right\} .
\end{equation*}
\end{definition}

One can show by induction that the iterated centralizers form an ascendig
sequence: $1=C_{G}^{0}\left( H\right) \leq C_{G}^{1}\left( H\right) \leq
...\leq G.$ In contrast with iterated centralizers $E_{k}$ envelopes form a
descending sequence such as $G=E_{0}\left( H\right) \geq E_{1}\left(
H\right) \geq ...\geq H.$

When $A=G$, the $k$th iterated centralizer of $G$ is more commonly known as $%
Z_{k}\left( G\right) $ and defined as follows:

\begin{definition}
Let $G$ be a group. Setting $Z_{0}\left( G\right) =\left\{ 1\right\} $, the $%
k$th center of $G$ is 
\begin{equation*}
Z_{k}\left( G\right) =\left\{ g\in G\mid \left[ g,G\right] \subseteq
Z_{k-1}\left( G\right) \right\}
\end{equation*}%
for all $k\geq 1$.
\end{definition}

\bigskip

Some of the basic relations between the iterated centralizers and iterated
centers are stated below:

\bigskip

\begin{fact}[{[1, Lemma 2.5]}]
\label{2.5}Let $A\leq B\leq C$ be groups and suppose that for all $j\leq k$
we have $C_{C}^{j}\left( A\right) =$ $Z_{j}\left( C\right) .$ Then

\begin{enumerate}
\item[(i)] $C_{C}^{j}\left( A\right) =C_{C}^{j}\left( B\right) =Z_{j}\left(
C\right) ,\forall j\leq k$

\item[(ii)] $C_{B}^{j}(A)=Z_{j}\left( B\right) =Z_{j}\left( C\right) \cap
B,\forall j\leq k$

\item[(iii)] $C_{B}^{k+1}(A)=C_{C}^{k+1}(A)\cap B,\forall j\leq k.$
\end{enumerate}
\end{fact}

\bigskip

\begin{fact}[{[1, Theorem 3.7, (2)] }]
\label{2.6}Let $G$ be an arbitrary group and $H$ a subgroup of $G.$ Then%
\begin{equation*}
C_{E_{k}(H)}^{j}(H)=Z_{j}\left( E_{k}(H)\right)
\end{equation*}%
for all $j\leq k$.
\end{fact}

\bigskip

In \cite{Cakmak}, the relation between any nilpotent subgroup of an
arbitrary group and its $E_{k}$ envelope is given.

\bigskip

\begin{fact}[\protect\cite{Cakmak}]
\label{3.2}\bigskip Let $G$ be a group and $H\leq G.$ If $H$ is a $k$%
-nilpotent subgroup, then the envelope $E_{k}\left( H\right) $ is also $k$%
-nilpotent.
\end{fact}

Hypercentral groups generalize nilpotent groups. These groups can be
characterized in terms of the transfinitely extended upper central series,
which is defined in the following manner.

\begin{definition}
Let $G$ be a group and $\alpha $ an ordinal, the terms $Z_{\alpha }\left(
G\right) $ of the upper central series of $G$ are defined by the usual rules 
\begin{equation*}
Z_{0}\left( G\right) =\left\{ 1\right\} \text{ \ and \ }Z_{\alpha +1}\left(
G\right) \diagup Z_{\alpha }\left( G\right) =Z\left( G\diagup Z_{\alpha
}\left( G\right) \right)
\end{equation*}%
together with the completeness condition 
\begin{equation*}
Z_{\lambda }\left( G\right) =\underset{\alpha <\lambda }{\cup }Z_{\alpha
}\left( G\right)
\end{equation*}%
where $\lambda $ is a limit ordinal. The transfinite upper central series
terminate with a subgroup called the hypercenter of $G.$ It sometimes
convenient to call $Z_{\alpha }\left( G\right) $ the $\alpha -$hypercenter
of $G$. A group $G$ is called hypercentral if $G=Z_{\alpha }\left( G\right) $
for some ordinal $\alpha $. The smallest such $\alpha $ is called the degree
of hypercentrality of $G.$
\end{definition}

\bigskip

The following is a formal definition of a group satisfying the descending
chain condition on centralizers.

\begin{definition}
Let $G$ be a group. If there exist no infinite sequence of subsets $%
A_{n}\subseteq G$ such that $C_{G}\left( A_{n}\right) \geq C_{G}\left(
A_{n+1}\right) $ for all $n\in 
\mathbb{N}
,$ then $G$ has the chain condition on centralizers and denoted as $%
\mathfrak{M}_{C}$. By elementary properties of centralizers the descending
chain condition on centralizers is equivalent to the ascending chain
condition on centralizers.
\end{definition}

\bigskip 

The following property of $\mathfrak{M}_{C}$-groups will be useful in the
paper. 

\bigskip 

\begin{fact}[{[2, Corollary 2.3]}]
\label{3.3}Let $G$ be a locally nilpotent $\mathfrak{M}_{C}$-group. Then $G$
is solvable.
\end{fact}

\section{\protect\bigskip Technical Definitions and Facts \label{3}}

In this section, we will introduce several technical notions and prove their
properties needed for our main result, Theorem \ref{3.1.11}. For this
purpose, we shall define the transfinite forms of $E_{k}$ definable
envelopes and iterated centralizers. After that, we will prove the
transfinite forms of the Facts \ref{2.5} and \ref{2.6}. For simplicity, we
will denote $E_{\alpha }\left( H\right) $ by $E_{\alpha }$ when the context
is clear.

\begin{definition}
\label{3.1.1}Let $H$ be any subgroup of the group $G$. Set $C_{G}^{0}\left(
H\right) =1$ and for $\alpha $ ordinal number, the $C_{G}^{\alpha }\left(
H\right) $ iterated centralizers are defined as follow:

\begin{enumerate}
\item[(i)] If $\alpha $ is a successor ordinal 
\begin{equation*}
C_{G}^{\alpha }\left( H\right) =\left\{ x\in \underset{\beta <\alpha }{\cap }%
N_{G}\left( C_{G}^{\beta }\left( H\right) \right) \mid \left[ x,H\right]
\subseteq C_{G}^{\alpha -1}\left( H\right) \right\} ,
\end{equation*}

\item[(ii)] If $\alpha $ is a limit ordinal%
\begin{equation*}
C_{G}^{\alpha }\left( H\right) =\underset{\beta <\alpha }{\cup }C_{G}^{\beta
}\left( H\right) .
\end{equation*}
\end{enumerate}
\end{definition}

\begin{definition}
\label{3.1.2}Let $H$ be any subgroup of the group $G$. Set $E_{0}\left(
H\right) =G$ and for $\alpha $ ordinal number, the $E_{\alpha }\left(
H\right) $ envelopes are defined as follow:
\end{definition}

\begin{enumerate}
\item[(i)] If $\alpha $ is a successor ordinal%
\begin{equation*}
E_{\alpha }\left( H\right) =\left\{ g\in E_{\alpha -1}\left( H\right) \mid 
\left[ g,C_{E_{\alpha -1}\left( H\right) }^{\alpha }\left( H\right) \right]
\leq C_{E_{\alpha -1}\left( H\right) }^{\alpha -1}\left( H\right) \right\} ,
\end{equation*}

\item[(ii)] If $\alpha $ is a limit ordinal 
\begin{equation*}
E_{\alpha }\left( H\right) =\underset{\beta <\alpha }{\cap }E_{\beta }\left(
H\right) .
\end{equation*}
\end{enumerate}

\bigskip When the subgroup is clear, we will shorten $E_{\alpha }\left(
H\right) $ to $E_{\alpha }$.

Now we will give a technical lemma that generalizes Fact \ref{2.5}.

\begin{lemma}
\label{3.1.3}Let $A\leq B\leq C$ be groups and $\lambda $ an ordinal such
that 
\begin{equation*}
C_{C}^{\alpha }\left( A\right) =Z_{\alpha }\left( C\right)
\end{equation*}%
for all $\alpha \leq \lambda $ ordinal numbers$.$ Then the following
equalities hold:

\begin{enumerate}
\item[(i)] $C_{C}^{\alpha }\left( A\right) =C_{C}^{\alpha }\left( B\right)
=Z_{\alpha }\left( C\right) $

\item[(ii)] $C_{B}^{\alpha }\left( A\right) =Z_{\alpha }\left( B\right)
=Z_{\alpha }\left( C\right) \cap B$

\item[(iii)] $C_{B}^{\lambda +1}\left( A\right) =C_{C}^{\lambda +1}\left(
A\right) \cap B.$
\end{enumerate}
\end{lemma}

\begin{proof}
We proceed by transfinite induction. For $\alpha =0,$ our claims are
trivial. Now we will show that the claims \textit{(i)} and \textit{(ii)} are
true both of the successor and limit ordinals. Suppose that the claims 
\textit{(i)} and \textit{(ii)} hold for all $\beta <\lambda $ ordinals.

\begin{enumerate}
\item[(i)] 
\end{enumerate}

\begin{itemize}
\item Let $\alpha $ be a successor ordinal, namely $\alpha =\beta +1$. By
the hypothesis of theorem, it is known that $C_{C}^{\beta +1}\left( A\right)
=Z_{\beta +1}\left( C\right) $ since $\beta +1\leq \lambda $ when $\beta
<\lambda $. From this fact and the induction hypothesis we have 
\begin{eqnarray*}
C_{C}^{\beta +1}\left( B\right) &=&\left\{ x\in \underset{\gamma <\beta +1}{%
\cap }N_{C}\left( C_{C}^{\gamma }\left( B\right) \right) \mid \left[ x,B%
\right] \subseteq C_{C}^{\beta }\left( B\right) \right\} \\
&=&\left\{ x\in \underset{\gamma <\beta +1}{\cap }N_{C}\left( Z_{\gamma
}\left( C\right) \right) \mid \left[ x,B\right] \subseteq Z_{\beta }\left(
C\right) \right\} \\
&=&\left\{ x\in C\mid \left[ x,B\right] \subseteq C_{C}^{\beta }\left(
A\right) \right\} \subseteq C_{C}^{\beta +1}\left( A\right) .
\end{eqnarray*}%
Then 
\begin{equation*}
C_{C}^{\beta +1}\left( A\right) =Z_{\beta +1}\left( C\right) =C_{C}^{\beta
+1}\left( C\right) \subseteq C_{C}^{\beta +1}\left( B\right) \subseteq
C_{C}^{\beta +1}\left( A\right) ,
\end{equation*}%
so $C_{C}^{\beta +1}\left( A\right) =C_{C}^{\beta +1}\left( B\right) $ is
obtained.

\item If $\alpha $ is a limit ordinal; by using Definition \ref{3.1.1} and
the previous step, the following equations are obtained: 
\begin{eqnarray*}
C_{C}^{\alpha }\left( A\right) &=&\underset{\beta <\alpha }{\cup }%
C_{C}^{\beta }\left( A\right) =\underset{\beta <\alpha }{\cup }C_{C}^{\beta
}\left( B\right) =C_{C}^{\alpha }\left( B\right) , \\
C_{C}^{\alpha }\left( A\right) &=&\underset{\beta <\alpha }{\cup }%
C_{C}^{\alpha }\left( A\right) =\underset{\beta <\alpha }{\cup }Z_{\beta
}\left( C\right) =Z_{\alpha }\left( C\right) .
\end{eqnarray*}%
Thus claim \textit{(i)} holds for all ordinal numbers.
\end{itemize}

\begin{enumerate}
\item[(ii)] 
\end{enumerate}

\begin{itemize}
\item If $\alpha $ is a successor ordinal, namely $\alpha =\beta +1$; from
Definition \ref{3.1.1} and the induction hypotesis, 
\begin{eqnarray*}
C_{B}^{\beta +1}\left( A\right) &=&\left\{ x\in \underset{\delta <\beta +1}{%
\cap }N_{B}\left( C_{B}^{\delta }\left( A\right) \right) \mid \left[ x,A%
\right] \subseteq C_{B}^{\beta }\left( A\right) \right\} \\
&=&\left\{ x\in \underset{\delta <\beta +1}{\cap }N_{B}\left( Z_{\delta
}\left( B\right) \right) \mid \left[ x,A\right] \subseteq Z_{\beta }\left(
B\right) \right\} \\
&=&\left\{ x\in B\mid \left[ x,A\right] \subseteq Z_{\beta }\left( B\right)
\right\}
\end{eqnarray*}%
and also 
\begin{eqnarray*}
C_{C}^{\beta +1}\left( A\right) \cap B &=&\left\{ x\in C\mid \left[ x,A%
\right] \subseteq C_{C}^{\beta }\left( A\right) \right\} \cap B \\
&=&\left\{ x\in C\mid \left[ x,A\right] \subseteq Z_{\beta }\left( C\right)
\right\} \cap B \\
&=&\left\{ x\in B\mid \left[ x,A\right] \subseteq Z_{\beta }\left( B\right)
\right\} =C_{B}^{\beta +1}\left( A\right) .
\end{eqnarray*}%
is written. So, we get 
\begin{equation*}
C_{B}^{\beta +1}\left( A\right) =C_{C}^{\beta +1}\left( A\right) \cap
B=Z_{\beta +1}\left( C\right) \cap B.
\end{equation*}%
From the first and last terms, one of the equation of claim (ii) is
obtained. On the other hand, it can be written $Z_{\beta +1}\left( C\right)
\cap B=C_{C}^{\beta +1}\left( B\right) \cap B$ by using claim \textit{(i).}
Then we have 
\begin{eqnarray*}
Z_{\beta +1}\left( C\right) \cap B &=&C_{C}^{\beta +1}\left( B\right) \cap
B=\left\{ x\in C\mid \left[ x,B\right] \subseteq C_{C}^{\beta }\left(
B\right) \right\} \cap B \\
&=&\left\{ x\in B\mid \left[ x,B\right] \subseteq Z_{\beta }\left( C\right)
\cap B\right\} \\
&=&\left\{ x\in B\mid \left[ x,B\right] \subseteq Z_{\beta }\left( B\right)
\right\} =Z_{\beta +1}\left( B\right) .
\end{eqnarray*}%
Thus claim \textit{(ii)} holds for successor ordinals.

\item If $\alpha $ is a limit ordinal; by Definition \ref{3.1.1}\textbf{\ }%
and induction hypothesis for claim \textit{(ii)} we have 
\begin{equation*}
C_{B}^{\alpha }\left( A\right) =\underset{\beta <\alpha }{\cup }C_{B}^{\beta
}\left( A\right) =\underset{\beta <\alpha }{\cup }Z_{\beta }\left( B\right)
=Z_{\alpha }\left( B\right)
\end{equation*}%
and 
\begin{eqnarray*}
C_{B}^{\alpha }\left( A\right) &=&\underset{\beta <\alpha }{\cup }%
C_{B}^{\beta }\left( A\right) =C_{B}^{\alpha }\left( A\right) =\underset{%
\beta <\alpha }{\cup }\left( Z_{\beta }\left( C\right) \cap B\right) \\
&=&\left( \underset{\beta <\alpha }{\cup }Z_{\beta }\left( C\right) \right)
\cap B=Z_{\alpha }\left( C\right) \cap B.
\end{eqnarray*}%
Considering these equalities claim \textit{(ii)} follows for limit ordinals.
Thus claim \textit{(ii)} holds for all ordinal numbers.
\end{itemize}

We now shall prove claim \textit{(iii)}. By using Definition \ref{3.1.1}%
\textbf{\ }and the truth of claim \textit{(ii)} we write%
\begin{eqnarray}
C_{B}^{\lambda +1}\left( A\right) &=&\left\{ x\in \underset{\gamma <\lambda
+1}{\cap }N_{B}\left( Z_{\gamma }\left( B\right) \right) \mid \left[ x,A%
\right] \subseteq Z_{\lambda }\left( B\right) \right\}  \label{C1} \\[1pt]
&=&\left\{ x\in B\mid \left[ x,A\right] \subseteq Z_{\lambda }\left(
B\right) \right\} .  \notag
\end{eqnarray}%
On the other hand, for the iterated centralizer $C_{C}^{\lambda +1}\left(
A\right) $%
\begin{eqnarray*}
C_{C}^{\lambda +1}\left( A\right) &=&\left\{ x\in \underset{\delta \leq
\lambda }{\cap }N_{C}\left( C_{C}^{\delta }\left( A\right) \right) \mid %
\left[ x,A\right] \subseteq C_{C}^{\lambda }\left( A\right) \right\} \\
&=&\left\{ x\in \underset{\delta \leq \lambda }{\cap }N_{C}\left( Z_{\delta
}\left( C\right) \right) \mid \left[ x,A\right] \subseteq Z_{\lambda }\left(
C\right) \right\} \\
&=&\left\{ x\in C\mid \left[ x,A\right] \subseteq Z_{\lambda }\left(
C\right) \right\} .
\end{eqnarray*}%
is obtained from Definition \ref{3.1.1}\textbf{\ }and claim \textit{(i)}. If
the intersection of the iterated centralizer $C_{C}^{\lambda +1}\left(
A\right) $ with group $B$ is taken and claim \textit{(ii)} is used, the
following equation is obtained: 
\begin{eqnarray}
C_{C}^{\lambda +1}\left( A\right) \cap B &=&\left\{ x\in C\mid \left[ x,A%
\right] \subseteq Z_{\lambda }\left( C\right) \right\} \cap B  \label{C.2} \\%
[1pt]
&=&\left\{ x\in B\mid \left[ x,A\right] \subseteq Z_{\lambda }\left(
B\right) \right\} .  \notag
\end{eqnarray}%
Thus the result follows from the equations \ref{C1} and \ref{C.2}.
\end{proof}

\bigskip

The following lemma is of general interest.

\begin{lemma}
\label{3.1.4}Let $\left( C_{\alpha }\right) _{\alpha <\lambda }$ be an
ascending sequence of nonempty subsets of a set $E$. If $x\in \underset{%
\alpha <\lambda }{\cup }C_{\alpha }$, then there exists a minimal $\beta $
such that $x\in C_{\beta }$ and $\beta <\lambda $ and $\beta $ is a
successor ordinal.
\end{lemma}

\begin{proof}
If $x\in \underset{\alpha <\lambda }{\cup }C_{\alpha },$ then $x\in C_{\beta
},$for at least one $\beta <\lambda .$ If $\beta $ is a limit ordinal, then
the set $\left\{ \beta \leq \lambda ,\text{ }\beta \text{ limit ordinal}\mid
x\in C_{\beta }\right\} $ has a minimal element, $\beta _{0}.$ Since $%
C_{\beta _{0}}=\underset{\delta <\beta _{0}}{\cup }C_{\delta }$, there
exists $\delta _{0}<\beta _{0}$ such that $x$ $\in C_{\delta _{0}}.$ By the
choice of $\beta _{0},$ $\delta _{0}$ is successor.
\end{proof}

\bigskip

We now prove the iterated centralizers in transfinite form also compose an
ascending chain.

\begin{lemma}
Let $\alpha $ be an ordinal number. Then $C_{G}^{\alpha }\left( H\right)
\leq C_{G}^{\lambda }\left( H\right) $ for $\lambda \geq \alpha .$
\end{lemma}

\begin{proof}
We will argue by transfinite induction. If $\lambda =\alpha ,$ our claim is
trivial. Now suppose that $\lambda $ be $\beta +1$ is successor ordinal.
Then there are two cases:

\begin{description}
\item[\textbf{Case 1}] $\beta $ is a successor ordinal

It will be sufficient to show that $C_{G}^{\beta }\left( H\right) \leq
C_{G}^{\beta +1}\left( H\right) $ since the claim holds for all ordinals
smaller than $\beta .$ By using the Definition \ref{3.1.1} 
\begin{eqnarray*}
C_{G}^{\beta +1}\left( H\right) &=&\left\{ x\in \underset{\alpha <\beta +1}{%
\cap }N_{G}\left( C_{G}^{\alpha }\left( H\right) \right) \mid \left[ x,H%
\right] \subseteq C_{G}^{\beta }\left( H\right) \right\} , \\
C_{G}^{\beta }\left( H\right) &=&\left\{ x\in \underset{\alpha <\beta }{\cap 
}N_{G}\left( C_{G}^{\alpha }\left( H\right) \right) \mid \left[ x,H\right]
\subseteq C_{G}^{\beta -1}\left( H\right) \right\}
\end{eqnarray*}%
Let $x\in C_{G}^{\beta }\left( H\right) $. Then $x\in N_{G}\left(
C_{G}^{\alpha }\left( H\right) \right) $ for all $\alpha $ such that $\alpha
<\beta $. Since $C_{G}^{\beta }\left( H\right) \leq $ $N_{G}\left(
C_{G}^{\beta }\left( H\right) \right) $, $x\in N_{G}\left( C_{G}^{\alpha
}\left( H\right) \right) $. Hence, $x\in \underset{\alpha <\beta +1}{\cap }%
N_{G}\left( C_{G}^{\alpha }\left( H\right) \right) .$ On the other hand,
considering our claim holds for all ordinals smaller than $\beta ,$ 
\begin{equation*}
\left[ x,H\right] \subseteq C_{G}^{\beta -1}\left( H\right) \subseteq
C_{G}^{\beta }\left( H\right)
\end{equation*}%
is obtained for $x\in C_{G}^{\beta }\left( H\right) .$ That conclusion shows
that $C_{G}^{\beta }\left( H\right) \leq C_{G}^{\beta +1}\left( H\right) .$

\item[Case 2] $\beta $ is a limit ordinal.

By Definition \ref{3.1.1}, $C_{G}^{\beta }\left( H\right) =\underset{\alpha
<\beta }{\cup }C_{G}^{\alpha }\left( H\right) $. If $x\in C_{G}^{\beta
}\left( H\right) ,$ then $x$ is an element of $C_{G}^{\gamma }\left(
H\right) $ at least for one element such that $\gamma \leq \beta .$ By Lemma%
\textbf{\ }\ref{3.1.4}, there is a successor ordinal $\gamma $ such that $%
\gamma <\lambda $ and $x\in C_{G}^{\gamma }\left( H\right) $ at least for
one $\gamma \leq \beta $ while $x\in C_{G}^{\beta }\left( H\right) $.
Moreover, by induction $C_{G}^{\gamma }\left( H\right) \leq C_{G}^{\beta
}\left( H\right) $ for $\gamma \leq \beta .$ Then from the definition of $%
C_{G}^{\gamma }\left( H\right) $ it can be written that $x\in N_{G}\left(
C_{G}^{\delta }\left( H\right) \right) $ for all $\delta ,$ such that $%
\delta <\gamma .$ By induction, if $\gamma \leq \delta <\lambda =\beta +1,$
then 
\begin{equation*}
x\in C_{G}^{\gamma }\left( H\right) \leq C_{G}^{\delta }\left( H\right) \leq
N_{G}\left( C_{G}^{\delta }\left( H\right) \right) \Rightarrow x\in
N_{G}\left( C_{G}^{\delta }\left( H\right) \right) .
\end{equation*}%
It remains to show the commutator condition to verify the inclusion $%
C_{G}^{\beta }\left( H\right) \leq C_{G}^{\beta +1}\left( H\right) $.
Considering Definition \ref{3.1.1}, the induction for $\gamma \leq \beta $
and $x\in C_{G}^{\gamma }\left( H\right) ,$%
\begin{equation*}
\left[ x,H\right] \subseteq C_{G}^{\gamma -1}\left( H\right) \subseteq
C_{G}^{\beta }\left( H\right) \Rightarrow \left[ x,H\right] \subseteq
C_{G}^{\beta }\left( H\right) .
\end{equation*}%
is obtained and our claim holds.
\end{description}

Finally if $\lambda $ is limit ordinal, we have 
\begin{equation*}
C_{G}^{\lambda }\left( H\right) =\underset{\beta <\lambda }{\cup }%
C_{G}^{\beta }\left( H\right) =\underset{\alpha \leq \beta <\lambda }{\cup }%
C_{G}^{\alpha }\left( H\right) \geq C_{G}^{\alpha }\left( H\right)
\end{equation*}%
by Definition \ref{3.1.1}. Then the result follows.
\end{proof}

\bigskip

The following lemma is a transfinite version of Fact \ref{2.6}.

\begin{lemma}
\label{3.1.6}Let $\lambda $ be an ordinal number. Then $C_{E_{\lambda
}(H)}^{\alpha }\left( H\right) =Z_{\alpha }\left( E_{\lambda }(H)\right) $\
for all ordinals such that $\alpha \leq \lambda .$
\end{lemma}

\begin{proof}
We proceed by transfinite induction on $\lambda $. When $\lambda =0,$ $%
\alpha =0$. So, our claim is trivial for $\lambda =0$.

Let $\lambda $ be a successor ordinal, i.e $\lambda =\beta +1.$ In
particular $\alpha \leq \beta $. It is known that $C_{E_{\beta }}^{\alpha
}\left( H\right) =Z_{\alpha }\left( E_{\beta }\right) $ by induction. So, it
will suffice to show the equality%
\begin{equation*}
C_{E_{\beta +1}}^{\alpha }\left( H\right) =Z_{\alpha }\left( E_{\beta
+1}\right)
\end{equation*}%
for $\alpha =\beta +1.$ Applying Lemma\textbf{\ }\ref{3.1.3}\textbf{\ }%
\textit{(ii)} to the $H\leq E_{\beta +1}\leq E_{\beta }$ triple we have 
\begin{equation*}
C_{E_{\beta +1}}^{\alpha }\left( H\right) =Z_{\alpha }\left( E_{\beta
+1}\right) \trianglelefteq E_{\beta +1}\Rightarrow C_{E_{\beta +1}}^{\alpha
}\left( H\right) \trianglelefteq E_{\beta +1}.
\end{equation*}%
On the other hand from Definition \ref{3.1.1}\textbf{\ }and\ Lemma \ref%
{3.1.3}\textbf{\ }\textit{(ii)} 
\begin{equation*}
\left[ Z_{\beta +1}\left( E_{\beta +1}\right) ,H\right] \leq Z_{\beta
}\left( E_{\beta +1}\right) =C_{E_{\beta +1}}^{\beta }\left( H\right)
\end{equation*}%
is obtained. Now we shall show bidirectional inclusion by using Definition %
\ref{3.1.1}\textbf{\ }and\ Lemma \ref{3.1.3} 
\begin{eqnarray*}
C_{E_{\beta +1}}^{\beta +1}\left( H\right) &=&\left\{ x\in \underset{\gamma
<\beta +1}{\cap }N_{E_{\beta +1}}\left( C_{E_{\beta +1}}^{\gamma }\left(
H\right) \right) \mid \left[ x,H\right] \subseteq C_{E_{\beta +1}}^{\beta
}\left( H\right) \right\} \\
&=&\left\{ x\in \underset{\gamma <\beta +1}{\cap }N_{E_{\beta +1}}\left(
Z_{\gamma }\left( E_{\beta +1}\right) \right) \mid \left[ x,H\right]
\subseteq Z_{\beta }\left( E_{\beta +1}\right) \right\} \\
&=&\left\{ x\in E_{\beta +1}\mid \left[ x,H\right] \subseteq Z_{\beta
}\left( E_{\beta +1}\right) \right\} .
\end{eqnarray*}%
So, we get 
\begin{equation*}
Z_{\beta +1}\left( E_{\beta +1}\right) \leq C_{E_{\beta +1}}^{\beta }\left(
H\right) \leq C_{E_{\beta +1}}^{\beta +1}\left( H\right) .
\end{equation*}%
Thus the inclusion $Z_{\beta +1}\left( E_{\beta +1}\right) \leq C_{E_{\beta
+1}}^{\beta +1}\left( H\right) $ is verified. We now will prove the reverse
inclusion. Considering Lemma \ref{3.1.3}\textbf{\ }\textit{(ii)} \textbf{\ }%
\begin{equation*}
\left[ C_{E_{\beta +1}}^{\beta +1}\left( H\right) ,E_{\beta +1}\right] =%
\left[ C_{E_{\beta }}^{\beta +1}\left( H\right) \cap E_{\beta +1},E_{\beta
+1}\right]
\end{equation*}%
is written. We will show this commutator is in $Z_{\beta }\left( E_{\beta
+1}\right) $ to verify the inclusion $C_{E_{\beta +1}}^{\beta +1}\left(
H\right) \leq Z_{\beta +1}\left( E_{\beta +1}\right) $. Considering the
definition of $E_{\beta +1}$ and Lemma \ref{3.1.3}\textbf{\ \textit{(i)}\ }%
\begin{equation*}
E_{\beta +1}=\left\{ x\in E_{\beta }\mid \left[ x,C_{E_{\beta }}^{\beta
+1}\left( H\right) \right] \leq C_{E_{\beta }}^{\beta }\left( H\right)
\right\} =\left\{ x\in E_{\beta }\mid \left[ x,C_{E_{\beta }}^{\beta
+1}\left( H\right) \right] \leq Z_{\beta }\left( E_{\beta }\right) \right\} .
\end{equation*}%
Since $E_{\beta +1}\leq E_{\beta }$, the commutator of $x\in E_{\beta +1}$
and $C_{E_{\beta }}^{\beta +1}\left( H\right) $ is in $Z_{\beta }\left(
E_{\beta }\right) $. So, by using this fact and Lemma \ref{3.1.3}\textbf{\ }%
\begin{eqnarray*}
\left[ C_{E_{\beta +1}}^{\beta +1}\left( H\right) ,E_{\beta +1}\right] &=&%
\left[ C_{E_{\beta }}^{\beta +1}\left( H\right) \cap E_{\beta +1},E_{\beta
+1}\right] \\
&\leq &Z_{\beta }\left( E_{\beta }\right) \cap E_{\beta +1}=Z_{\beta }\left(
E_{\beta +1}\right) \\
&\Rightarrow &C_{E_{\beta +1}}^{\beta +1}\left( H\right) \leq Z_{\beta
}\left( E_{\beta +1}\right) \leq Z_{\beta +1}\left( E_{\beta +1}\right)
\end{eqnarray*}%
is obtained. So we are done. Then, the claim holds for $\lambda =\beta +1$
successor ordinal.

Finally let $\lambda $ be a limit ordinal. While $\alpha <\lambda ,$ we have 
$C_{E_{\alpha }}^{\beta }\left( H\right) =Z_{\beta }\left( E_{\alpha
}\right) $ for all $\beta \leq \alpha $ from inductive hypothesis. Then we
can apply Lemma \ref{3.1.3}\textbf{\ }to $H\leq $ $E_{\lambda }\leq
E_{\alpha }$ subgroups for $\alpha \leq \lambda $. So, we get%
\begin{equation}
C_{E_{\lambda }}^{\beta }\left( H\right) =Z_{\beta }\left( E_{\lambda
}\right)  \label{C3}
\end{equation}%
for $\beta \leq \alpha <\lambda $. When $\beta =\lambda $, consideringly
Definition \ref{3.1.1}\textbf{,} and the equality\textbf{\ }\ref{C3} the
following result is obtained:%
\begin{equation*}
C_{E_{\lambda }}^{\lambda }\left( H\right) =\underset{\beta <\lambda }{\cup }%
C_{E_{\lambda }}^{\beta }\left( H\right) =\underset{\beta <\lambda }{\cup }%
Z_{\beta }\left( E_{\lambda }\right) =Z_{\lambda }\left( E_{\lambda }\right)
.
\end{equation*}%
Thus, our claim follows for ordinal numbers.
\end{proof}

\bigskip

In the rest of this section, we will prove a special ascendance property of
the $E_{\alpha }$-envelopes.

\begin{lemma}
\label{3.1.12}Let $\alpha $ be an ordinal number. Then $Z_{\alpha }\left(
E_{\alpha }\right) \leq E_{\lambda }$\ for $\alpha \leq \lambda $.
\end{lemma}

\begin{proof}
For $\alpha =\lambda $, since $Z_{\alpha }\left( E_{\alpha }\right) \leq
E_{\alpha }$, our claim is clear. Suppose that $\alpha <\lambda $. Let $%
\lambda $ be a successor ordinal, i.e $\lambda =\beta +1$. By induction, for 
$\alpha \leq \beta $ we have $Z_{\alpha }\left( E_{\alpha }\right) \leq
E_{\beta }$. It is known that $E_{\beta }\leq E_{\alpha }$ for $\alpha \leq
\beta $ by Definition \ref{3.1.2}\textbf{.} Applying Lemma\textbf{\ }\ref%
{3.1.3}\textbf{\ }\textit{(ii)} to the $H\leq E_{\beta }\leq E_{\alpha }$
groups, \textbf{\ }%
\begin{equation*}
Z_{\alpha }\left( E_{\beta }\right) =Z_{\alpha }\left( E_{\alpha }\right)
\cap E_{\beta }
\end{equation*}%
is obtained. By using the induction hypothesis we have 
\begin{equation*}
Z_{\alpha }\left( E_{\beta }\right) =Z_{\alpha }\left( E_{\alpha }\right)
\cap E_{\beta }=Z_{\alpha }\left( E_{\alpha }\right) .
\end{equation*}%
On the other hand considering the facts that; $Z_{\alpha }\left( E_{\beta
}\right) \leq Z_{\beta }\left( E_{\beta }\right) $ for $\alpha \leq \beta $
and $Z_{\beta }\left( E_{\beta }\right) \trianglelefteq E_{\beta }$ we get 
\begin{equation*}
Z_{\alpha }\left( E_{\alpha }\right) =Z_{\alpha }\left( E_{\alpha }\right)
\cap E_{\beta }=Z_{\alpha }\left( E_{\beta }\right) \leq Z_{\beta }\left(
E_{\beta }\right) .
\end{equation*}%
It remains to show that 
\begin{equation*}
\left[ Z_{\alpha }\left( E_{\alpha }\left( H\right) \right) ,C_{E_{\beta
}(H)}^{\beta +1}\left( H\right) \right] \leq Z_{\beta }\left( E_{\beta
}\left( H\right) \right)
\end{equation*}%
to prove the claim for successor ordinals. Since $C_{E_{\beta }(H)}^{\beta
+1}\left( H\right) \leq E_{\beta }\left( H\right) $ and $Z_{\alpha }\left(
E_{\alpha }\right) \leq E_{\beta }$ by induction, we find the following
inclusion 
\begin{equation*}
\left[ Z_{\alpha }\left( E_{\alpha }\right) ,C_{E_{\beta }}^{\beta +1}\left(
H\right) \right] \subseteq Z_{\alpha }\left( E_{\alpha }\right) \subseteq
Z_{\beta }\left( E_{\beta }\right) .
\end{equation*}%
According to this $Z_{\alpha }\left( E_{\alpha }\right) \leq E_{\beta +1}.$
Then our claim holds for $\lambda =\beta +1$ successor ordinal.

For $\lambda $ limit ordinal, let $Z_{\alpha }\left( E_{\alpha }\right) $ be
a subgroup of $E_{\beta }$ for $\alpha \leq \beta <\lambda $. From
Definition \ref{3.1.2}, 
\begin{equation*}
E_{\lambda }=\underset{\beta <\lambda }{\cap }E_{\beta }=\underset{\alpha
\leq \beta <\lambda }{\cap }E_{\beta }\geq Z_{\alpha }\left( E_{\alpha
}\right)
\end{equation*}%
is written. Thus, the result follows from the first and last terms.
\end{proof}

\bigskip

\begin{corollary}
\label{3.1.13}Let $\alpha $ be an ordinal number. Then 
\begin{equation*}
Z_{\alpha }\left( E_{\alpha }\right) \leq Z_{\lambda }\left( E_{\lambda
}\right)
\end{equation*}%
for $\alpha \leq \lambda $.
\end{corollary}

\begin{proof}
When $\alpha \leq \lambda $, $E_{\lambda }\leq E_{\alpha }$. By Lemma\textbf{%
\ }\ref{3.1.3}\textbf{\ }and \ref{3.1.12} we have 
\begin{equation*}
Z_{\lambda }\left( E_{\lambda }\right) \geq Z_{\alpha }\left( E_{\lambda
}\right) =Z_{\alpha }\left( E_{\alpha }\right) \cap E_{\lambda }=Z_{\alpha
}\left( E_{\alpha }\right) .
\end{equation*}%
So, we are done.
\end{proof}

\bigskip

\section{Hypercentral Subgroups}

In this section, we apply the technical tools developed in the previous
sections to the analysis of the $E_{\alpha }$-envelopes of hypercentral
subgroups of various classes of groups. This allows us to draw conclusions
on hypercentral subgroups of $\mathfrak{M}_{c}$-groups (Corollary \ref{3.1.9}%
) and prove a general finiteness result (Theorem \ref{3.1.11}) which is the
main conclusion of the paper.

\bigskip

\begin{proposition}
\label{3.1.8}Let $G$ be a group and $H\leq G.$ Let $\alpha $ be an ordinal.
Then;

\begin{enumerate}
\item[(i)] If $H$ is an $\left( \alpha +1\right) $-hypercentral subgroup, $%
E_{\alpha +1}\left( H\right) $ is also $\left( \alpha +1\right) $%
-hypercentral.

\item[(ii)] If $H$ is an $\alpha $-hypercentral subgroup, then $E_{\alpha
+1}\left( H\right) $ is at most $\alpha $-hypercentral.
\end{enumerate}
\end{proposition}

\begin{proof}

\begin{enumerate}
\item[(i)] Let $H$ be an $\left( \alpha +1\right) $-hypercentral subgroup.
By using the second isomorphism theorem and Lemma \ref{3.1.3}\textbf{\ }%
\textit{(ii)}\textbf{\ }for the triple $H\leq $ $H\leq E_{\alpha }$ 
\begin{equation*}
HZ_{\alpha }\left( E_{\alpha }\right) \diagup Z_{\alpha }\left( E_{\alpha
}\right) \cong H\diagup H\cap Z_{\alpha }\left( E_{\alpha }\right) =H\diagup
Z_{\alpha }\left( H\right)
\end{equation*}%
is written. Since $H$ is $\left( \alpha +1\right) $-hypercentral, $H\diagup
Z_{\alpha }\left( H\right) $ is abelian. By the second isomorphism theorem,
Lemma \ref{3.1.12} and applying Lemma \ref{3.1.3} \textit{(ii)} to $H\leq
E_{\alpha +1}\left( H\right) \leq E_{\alpha }\left( H\right) $ triple 
\begin{equation*}
E_{\alpha +1}\diagup Z_{\alpha }\left( E_{\alpha }\right) \cong E_{\alpha
+1}\diagup E_{\alpha +1}\cap Z_{\alpha }\left( E_{\alpha }\right) =E_{\alpha
+1}\diagup Z_{\alpha }\left( E_{\alpha +1}\right)
\end{equation*}%
is obtained. On the other hand, by the definition of $E_{\alpha +1}$ and $%
C_{E_{\alpha }}^{\alpha +1}\left( H\right) $, we have 
\begin{equation*}
E_{\alpha +1}\diagup Z_{\alpha }\left( E_{\alpha }\right) =C_{E_{\alpha
}\diagup Z_{\alpha }\left( E_{\alpha }\right) }\left( C_{E_{\alpha }\diagup
Z_{\alpha }\left( E_{\alpha }\right) }\left( HZ_{\alpha }\left( E_{\alpha
}\right) \diagup Z_{\alpha }\left( E_{\alpha }\right) \right) \right) .
\end{equation*}%
Since $H\diagup Z_{\alpha }\left( H\right) $ is abelian and $HZ_{\alpha
}\left( E_{\alpha }\right) \diagup Z_{\alpha }\left( E_{\alpha }\right)
\cong H\diagup Z_{\alpha }\left( H\right) $, the following group 
\begin{equation*}
E_{\alpha +1}\diagup Z_{\alpha }\left( E_{\alpha }\right) \cong E_{\alpha
+1}\diagup Z_{\alpha }\left( E_{\alpha +1}\right) .
\end{equation*}%
is also abelian from Fact \ref{3.1}\textbf{. }Thus the subgroup $E_{\alpha
+1}$ is at most $\left( \alpha +1\right) $-hypercentral. But at the same
time $E_{\alpha +1}$ is exactly $\left( \alpha +1\right) $-hypercentral
since the hypercentrality class of a group can not be smaller than the
hypercentrality class of subgroup.

\item[(ii)] When $\alpha $ is a limit ordinal $H\diagup Z_{\alpha }\left(
H\right) =1\ $since\ $Z_{\alpha }\left( H\right) =H.$ So, we have 
\begin{equation*}
HZ_{\alpha }\left( E_{\alpha }\right) \diagup Z_{\alpha }\left( E_{\alpha
}\right) \cong H\diagup Z_{\alpha }\left( H\right) =1.
\end{equation*}%
In addition to (i), by using Definitions \ref{3.1.1}\textbf{, }\ref{3.1.2}
and Lemma \ref{3.1.12} the following abelian group is obtained: 
\begin{eqnarray*}
E_{\alpha +1}\diagup Z_{\alpha }\left( E_{\alpha +1}\right) &\cong
&E_{\alpha +1}\diagup Z_{\alpha }\left( E_{\alpha }\right) \\
&=&C_{E_{\alpha }\diagup Z_{\alpha }\left( E_{\alpha }\right) }\left(
C_{E_{\alpha }\diagup Z_{\alpha }\left( E_{\alpha }\right) }\left(
HZ_{\alpha }\left( E_{\alpha }\right) \diagup Z_{\alpha }\left( E_{\alpha
}\right) \right) \right) \\
&=&C_{E_{\alpha }\diagup Z_{\alpha }\left( E_{\alpha }\right) }\left(
C_{E_{\alpha }\diagup Z_{\alpha }\left( E_{\alpha }\right) }\left( 1\right)
\right) =Z\left( E_{\alpha }\diagup Z_{\alpha }\left( E_{\alpha }\right)
\right) .
\end{eqnarray*}%
Then, $E_{\alpha +1}$ is at most $\left( \alpha +1\right) $-hypercentral
subgroup since $E_{\alpha +1}\diagup Z_{\alpha }\left( E_{\alpha +1}\right) $
is abelian. Besides applying\textbf{\ }Lemma \ref{3.1.3} \textit{(ii)}%
\textbf{\ }respectively to the subgroups $H\leq $ $H\leq E_{\beta }$ and $%
H\leq $ $H\leq E_{\beta +1}$ 
\begin{eqnarray*}
Z_{\beta }\left( H\right) &=&Z_{\beta }\left( E_{\beta }\right) \cap H, \\
Z_{\beta +1}\left( H\right) &=&Z_{\beta +1}\left( E_{\beta +1}\right) \cap H,
\end{eqnarray*}%
is written for $\beta <\alpha $. When $\beta =\alpha $, from the hypothesis
of proposition and the fact that $E_{\alpha +1}$ is at most $\left( \alpha
+1\right) $-hypercentral, we have 
\begin{equation}
Z_{\alpha }\left( H\right) =Z_{\alpha +1}\left( H\right) \Rightarrow
Z_{\alpha }\left( E_{\alpha }\right) =Z_{\alpha +1}\left( E_{\alpha
+1}\right) =E_{\alpha +1}.  \label{C4}
\end{equation}%
Considering the equality \ref{C4}, we get \textbf{\ } 
\begin{equation}
Z_{\alpha }\left( E_{\alpha +1}\right) =Z_{\alpha }\left( E_{\alpha }\right)
\cap E_{\alpha +1}=Z_{\alpha }\left( E_{\alpha }\right) =Z_{\alpha +1}\left(
E_{\alpha +1}\right) =E_{\alpha +1}.  \label{C5}
\end{equation}%
Here Lemma \ref{3.1.3}\textbf{\ }\textit{(ii)} was applied to $H\leq
E_{\alpha +1}\leq E_{\alpha }$ subgroups for the first equation, while Lemma %
\ref{3.1.12} and the fact that at most $\left( \alpha +1\right) $%
-hypercentrality of $E_{\alpha +1}\left( H\right) $ were used respectively
for the second and third equations. So $Z_{\alpha }\left( E_{\alpha
+1}\right) =E_{\alpha +1}$.
\end{enumerate}
\end{proof}

\bigskip

We now prove a corollary of Proposition \ref{3.1.8} for an $\mathfrak{M}_{c}$%
-group.

\begin{corollary}
\label{3.1.9}Let $G$ be an $\mathfrak{M}_{c}$-group and $H$ be an $\alpha $%
-hypercentral subgroup of $G$. Then $E_{\alpha +1}\left( H\right) $ is
solvable.
\end{corollary}

\begin{proof}
B\textbf{y }Proposition \ref{3.1.8}\textbf{\ }\textit{(ii)} $E_{\alpha
+1}\left( H\right) $ is hypercentral. Since any hypercentral group is
locally nilpotent and the class $\mathfrak{M}_{c}$\ is closed under the
formation of subgroups, $E_{\alpha +1}\left( H\right) $ is a locally
nilpotent $\mathfrak{M}_{c}$-group. By Fact \ref{3.3} $E_{\alpha +1}\left(
H\right) $ envelope is solvable.
\end{proof}

\bigskip

The conclusion of this corollary is the best possible in this direction.
Indeed, the following example shows that the envelope of an hypercentral
subgroup of an $\mathfrak{M}_{c}$-group be non nilpotent:

\bigskip

\textit{Let }$G$ denote $GL_{2}\left( 
\mathbb{C}
\right) $\textit{\ that has }$\mathfrak{M}_{c}$\textit{-property and the
hypercentral subgroup from the successor ordinal degree of }$G$%
\begin{equation*}
H_{\infty }=\left\langle 
\begin{pmatrix}
\lambda _{1} & 0 \\ 
0 & \lambda _{2}%
\end{pmatrix}%
\mid \lambda _{i}^{2^{n}}=1,\text{ }\lambda _{i}\in 
\mathbb{C}
,\text{\ }i,n\in 
\mathbb{N}
\right\rangle \leq G
\end{equation*}%
\textit{The envelopes of \bigskip }$H_{\infty }$\textit{\ subgroup are
determined as follows:}

\begin{itemize}
\item $E_{0}\left( H_{\infty }\right) =E_{1}\left( H_{\infty }\right) =G,$
\end{itemize}

\bigskip

\begin{itemize}
\item \textit{For }$n\in 
\mathbb{N}
,$ $n\neq 0,1$\textit{; }$E_{n}\left( H_{\infty }\right) =\left\langle 
\begin{pmatrix}
x & 0 \\ 
0 & y%
\end{pmatrix}%
\mid x,y\in 
\mathbb{C}
^{\ast }\right\rangle \rtimes \left\langle 
\begin{pmatrix}
0 & 1 \\ 
1 & 0%
\end{pmatrix}%
\right\rangle ,$
\end{itemize}

\bigskip

\begin{itemize}
\item $E_{\omega }\left( H_{\infty }\right) =\left\langle 
\begin{pmatrix}
x & 0 \\ 
0 & y%
\end{pmatrix}%
\mid x,y\in 
\mathbb{C}
^{\ast }\right\rangle \rtimes \left\langle 
\begin{pmatrix}
0 & 1 \\ 
1 & 0%
\end{pmatrix}%
\right\rangle ,$
\end{itemize}

\bigskip

\begin{itemize}
\item $E_{\omega +1}\left( H_{\infty }\right) =\left\langle 
\begin{pmatrix}
x & 0 \\ 
0 & x1^{\left( \frac{1}{2}\right) ^{n}}%
\end{pmatrix}%
\mid x\in 
\mathbb{C}
^{\ast },n\in 
\mathbb{N}
^{\ast }\right\rangle \rtimes \left\langle 
\begin{pmatrix}
0 & 1 \\ 
1 & 0%
\end{pmatrix}%
\right\rangle ,$
\end{itemize}

\textit{where\ }$1^{\left( \frac{1}{2}\right) ^{n}}$\textit{\ denotes the }$%
2^{n}-th$\textit{\ roots of\ }$1.$ $E_{\omega +1}\left( H_{\infty }\right) $%
\textit{\ is a hypercentral subgroup from }$\left( \omega +1\right) -th$%
\textit{\ degree by Proposition \ref{3.1.8} and so it is locally nilpotent.
At the same time }$E_{\omega +1}\left( H_{\infty }\right) $\textit{\ is }an $%
\mathfrak{M}_{c}$-group\textit{\ since }$G$ is an $\mathfrak{M}_{c}$\textit{%
-group. By Fact \ref{3.3}} $E_{\omega +1}\left( H_{\infty }\right) $\textit{%
\ is solvable. But it is not nilpotent. }

\bigskip

Now we shall prove the main result of the paper, a finiteness condition that
extends Corollary 3.1.2 of \cite{Cakmak}:

\bigskip

\begin{theorem}
\label{3.1.11}Let $\alpha $ be a limit ordinal, $G$ a group and $H$ an $%
\alpha $-hypercentral subgroup of $G$. Then $E_{\alpha +1}\left( H\right)
=E_{\lambda }\left( H\right) $\ for all ordinals $\lambda $\ such that $%
\alpha +1\leq \lambda .$
\end{theorem}

\begin{proof}
We proceed by transfinite induction on $\lambda $. For $\lambda =\alpha +1,$
the claim is trivially satisfied. If $\lambda $ is successor ordinal
strictly bigger than $\alpha +1$, i.e $\lambda =\beta +1,$ it is known that
the claim is true for all $\alpha +1\leq \beta $ ordinals by induction%
\textbf{.} So, the following sequence of equalities holds:%
\begin{eqnarray*}
E_{\beta +1} &=&\left\{ x\in E_{\beta }\mid \left[ x,C_{E_{\beta }}^{\beta
+1}\left( H\right) \right] \leq C_{E_{\beta }}^{\beta }\left( H\right)
\right\} \\
&=&\left\{ x\in E_{\beta }\mid \left[ x,C_{E_{\beta }}^{\beta +1}\left(
H\right) \right] \leq Z_{\beta }\left( E_{\beta }\right) \right\} \\
&=&\left\{ x\in E_{\beta }\mid \left[ x,C_{E_{\beta }}^{\beta +1}\left(
H\right) \right] \leq E_{\beta }\right\} =E_{\beta }=E_{\alpha +1},
\end{eqnarray*}%
using also Lemma \ref{3.1.6} and Proposition \ref{3.1.8}\textbf{\ }\textit{%
(ii)}. If $\lambda $ limit is limit ordinal, by using the facts that the $%
E_{\lambda }$ form a descending chain and $E_{\alpha +1}=E_{\beta }$ for all
ordinals such that $\alpha +1\leq \beta ,$ we get 
\begin{equation*}
E_{\lambda }=\underset{\beta <\lambda }{\cap }E_{\beta }=\underset{\alpha
+1\leq \beta <\lambda }{\cap }E_{\alpha +1}=E_{\alpha +1}.
\end{equation*}%
Thus the result follows for all ordinals $\lambda $ such that $\alpha +1\leq
\lambda $.
\end{proof}

\bigskip

\section{Acknowledgements}

The author is deeply grateful to her Ph.D. supervisors Tuna Alt\i nel and
Erdal Karaduman for their valuable comments and suggestions all along the
way. She would also like to acknowledge the warm hospitalty of Universit\'{e}
Claude Bernard Lyon-1 where the results of the article were obtained.

\bigskip

\end{document}